\documentclass{amsart}
\usepackage{amscd,enumerate}

\newcommand*{\clo}{\overline}

\newcommand*{\R}{\ensuremath{\mathbb{R}}}

\renewcommand*{\epsilon}{\varepsilon}

\newtheorem{theorem}{Theorem}

\newtheorem{corollary}[theorem]{Corollary}

\newtheorem{proposition}[theorem]{Proposition}
\newtheorem{conjecture}[theorem]{Conjecture}

\newtheorem{example}[theorem]{Example}

\begin{document}

\title{Which connected spaces have a quotient homeomorphic to an arc}
\author{Micha{\l} Ryszard W\'ojcik}
\maketitle

\begin{abstract}
We announce and examine the conjecture that each infinite connected normal Hausdorff space has a quotient homeomorphic to the unit interval, shown to be true with the additional assumption of compactness or local connectedness. Some connected subsets of the plane are considered to show the difficulties involved in developing a general argument.
\end{abstract}

\section{Introduction}
We say that a topological space collapses to an arc if it has a quotient homeomorphic to the unit interval.
Let us demonstrate that any nondegenerate Hausdorff continuum collapses to an arc.
Indeed, let $X$ be an infinite connected compact normal Hausdorff space and let $a,b\in X$ be arbitrary distinct points. Since $X$ is normal, there is a Urysohn function for $a$ and $b$ --- a continuous function $f\colon X\to[0,1]$ such that $f(a)=0$ and $f(b)=1$. Since $X$ is connected, $f(X)=[0,1]$. Since $X$ is compact, $f$ maps closed sets onto closed sets and so it is a quotient mapping. This means that the fibers of $f$ form a quotient space that is homeomorphic to $[0,1]$. This observation motivates the following conjecture.
\begin{conjecture}
Each infinite connected normal Hausdorff space collapses to an arc (has a quotient homeomorphic to the unit interval).
\end{conjecture}

We give a partial affirmative solution for locally connected spaces.
Our main result is that every continuous surjection onto the unit interval defined on a connected locally connected space is quotient. So, if the domain is additionally assumed to be normal and Hausdorff, it collapses to an arc by the same argument as above.

We show examples of connected subsets of the plane with continuous surjections onto the unit interval (Urysohn functions) that are not quotient. However, these spaces collapse to an arc by a careful choice of a quotient Urysohn function which takes into account the particular way these spaces are embedded into the plane. This demonstrates the difficulty of obtaining an appropriate Urysohn function in the generic case of an abstract space that is assumed to satisfy the topological axioms of an infinite connected normal Hausdorff space.

If a normal space has a closed subspace that collapses to an arc then the quotient Urysohn function defined on this closed subspace extends to the whole space and is still quotient. This shows in particular that a counterexample cannot contain a connected compact infinite subset. So a possible counterexample is a connected punctiform space that is not locally connected.

These observations provide an outline for a fundamental research project into the nature of connectedness. If the conjecture turns out to be true it will mean that the unit interval is somehow present whenever appropriate topological axioms are satisfied. On the other hand, a counterexample will exhibit a fascinating pathology and broaden our intuition of a connected space. Even the task of looking at connected, not locally connected punctiform spaces and showing on a case by case basis that each of them collapses to an arc may be an exciting and educating experience sharpening our notion of a connected space.

Finally, let us note that we have discovered a generalization of arcwise connectedness, for if $a,b\in X$ are the endpoints of an arc $[a,b]\subset X$ embedded in the normal Hausdorff space $X$, then there is a homeomorphism $f\colon[a,b]\to[0,1]$ which can be extended to a continuous quotient surjection $f\colon X\to[0,1]$ and thus $X$ collapses to an arc in such a way that $a$ and $b$ lie in different fibers.
It would be interesting to find a space which collapses to an arc but has two points which end up in the same fiber whenever a continuous quotient surjection onto the unit interval is defined on the space.

\section{The results}

\begin{proposition}\label{quotientfibers}
Let $X,Y$ be topological spaces.
Let $f\colon X\to Y$ be a surjection.
Let $Z=\{f^{-1}(y)\colon y\in Y\}$ be the collection of all fibers of $f$.
Let $\mathcal G=\{V\subset Z\colon\bigcup V$ is open in $X\}$.
To each point $y\in Y$ we assign its fiber $\phi(y)=f^{-1}(y)\in Z$.
Then the following conditions are equivalent:
\begin{enumerate}
\item[(1)] $f^{-1}(G)\text{\ is open in }X\implies G\text{ is open in }Y$,
\item[(2)] $\phi\colon Y\to Z\text{\ is continuous}$.
\end{enumerate}
Moreover, the following conditions are equivalent:
\begin{enumerate}
\item[(3)] $G\text{ is open in }Y\implies f^{-1}(G)\text{\ is open in }X$,
\item[(4)] $\phi^{-1}\colon Z\to Y\text{\ is continuous}$.
\end{enumerate}
In all, $\phi$ is a homeomorphism between $Y$ and $Z$ iff $f$ is a continuous quotient map.
Moreover, a space has a quotient homeomorphic to the unit interval iff it admits a continuous quotient surjection onto the unit interval.
\end{proposition}

\begin{proposition}\label{closedopenquotient}
If $X$ is compact, $Y$ is Hausdorff, $f\colon X\to Y$ is a continuous surjection, then $f$ is a closed quotient mapping.
\end{proposition}
\begin{proof}
Let $E$ be closed in $X$. Since $X$ is compact, $E$ is compact. Since $f$ is continuous, $f(E)$ is compact. Since $Y$ is Hausdorff, $f(E)$ is closed in $Y$.
Let $G$ be open in $X$ with $G=f^{-1}(f(G))$. Then $Y\setminus f(G)=f(X\setminus G)$ is closed in $Y$. So $f(G)$ is open in $Y$.
\end{proof}

\begin{theorem}
Every nontrivial connected compact Hausdorff space has a quotient homeomorphic to an arc.
\end{theorem}
\begin{proof}
We have $a,b\in X,a\not=b$.
$X$ is normal being compact Hausdorff.
By Urysohn's lemma there is a continuous function $f\colon X\to[0,1]$ such that $f(a)=0$ and $f(b)=1$.
Since $X$ is connected, $f(X)=[0,1]$ .
Let $X/f$ be the quotient space formed by the fibers of $f$.
By Proposition \ref{closedopenquotient}, $f$ is quotient.
By Theorem \ref{quotientfibers}, $X/f$ is homeomorphic to $[0,1]$.
\end{proof}

\begin{theorem}\label{locallyconnected}
If $X$ is a connected, locally connected space,
$f\colon X\to[0,1]$ is a continuous surjection, then $f$ is quotient.
\end{theorem}
\begin{proof}
For each $r\in(0,1]$ let $B_r=f^{-1}([0,r))$ and $\clo{B}_r=f^{-1}([0,r])$.
We have $a,b\in X$ with $f(a)=0$ and $f(b)=1$.
Let $G$ be open in $X$ with $G=f^{-1}(f(G))$.
To show that $f(G)$ is open in $[0,1]$ take any $y\in f(G)$.
Since $X$ is connected, the open set $B_y$ cannot be closed because $a\in B_y$ and $b\not\in B_y$.
We get $u\in\clo{B_y}\setminus B_y$.
It follows that $f(u)=y\in f(G)$, so $u\in f^{-1}(f(G))\subset G$.
Since $X$ is locally connected, we get an open connected subset $U\subset X$ such that $u\in U\subset G$.
There is a point $x\in U\cap B_y$.
Let $r=f(x)$. We claim that $(r,y]\subset f(U)$.
Indeed, take any $\epsilon\in(r,y)$.
Then the set $U=(U\cap\clo{B}_\epsilon)\cup(U\setminus{B}_\epsilon)$
is a union of two relatively closed subsets.
They are both nonempty because $x\in U\cap\clo{B}_\epsilon$
and $u\in U\setminus{B}_\epsilon$.
Since $U$ is connected, we get $v\in U\cap\clo{B}_\epsilon\setminus{B}_\epsilon$,
so $f(v)=\epsilon$ and $v\in U$. Thus $\epsilon\in f(U)$.
It follows that $(r,y]\subset f(U)\subset f(G)$.
If $y=1$ the proof is finished. Otherwise, argue similarly starting with the observation that
$\{x\in X\colon f(x)>y\}$ cannot be closed to conclude that $[y,y+\delta)\subset f(G)$ for some $\delta>0$.
\end{proof}

\begin{theorem}
Let $X$ be a nontrivial connected, locally connected normal space.
Then $X$ has a quotient homeomorphic to an arc.
\end{theorem}
\begin{proof}
We have two distinct points $a,b\in X$. By Urysohn's lemma there is a continuous surjection $f\colon X\to[0,1]$ with $f(a)=0$ and $f(b)=1$. By Theorem \ref{locallyconnected}, $f$ is quotient.
By Theorem \ref{quotientfibers}, the quotient space formed by the fibers of $f$ is homeomorphic to $[0,1]$.
\end{proof}

\begin{corollary}[cf. \cite{continuityconnected}]
If $Y$ is a topological space and $f\colon[0,1]\to Y$ has a connected, locally connected graph, then $f$ is continuous.
\end{corollary}
\begin{proof}
Let $X$ be the graph of $f$ with the induced topology from $[0,1]\times Y$.
By assumption, $X$ is a connected locally connected space.
Then $\phi\colon X\to[0,1]$ given by $\phi(x,f(x))=x$ is a continuous surjection.
By Theorem \ref{locallyconnected}, $\phi$ is quotient.
Since $\phi$ is a bijection, it is an open mapping and thus $f$ is continuous.
\end{proof}

If $g\colon[0,1]\to Y$ is discontinuous with a connected graph then the graph cannot be locally connected. Let $X$ be the graph of $g$ and let $f\colon X\to[0,1]$ be given by $f(x,g(x))=x$. Then $f\colon X\to[0,1]$ is a continuous surjection that is not quotient because $g$ is discontinuous. Such examples show that the assumption that $X$ is locally connected cannot be dropped in Theorem \ref{locallyconnected}.

Let $(X,d)$ be a metric space, $a\in X$.
Let $f\colon X\to\R$ be given by $f(x)=d(x,a)$.
We already know that if $X$ is compact then $f$ is quotient.
However, $f$ need not be open even if $X$ is compact.
Moreover, $f$ need not be quotient if $X$ fails to be connected.
Indeed, let $X=\{(x,y)\in\R^2\colon x^2+y^2=1\vee(y=0\wedge 0\le x\le 1\}$ with the euclidean metric.
Let $a=(0,0)$ and let $G=B_d(a,1)$. Then $f(G)=\{1\}$, which is not open in $f(X)=[0,1]$.
Let us consider $Y=X\setminus\{(1,0)\}$ with the same metric. The space $Y$ is no longer connected and the function $f$ is not even quotient because $G=\{x\in Y\colon d(x,a)=1\}$ is now open in $Y$ and $G=f^{-1}(f(G))$ but $f(G)=\{1\}$ is not open in $f(Y)=[0,1]$.

\begin{example}
A connected metric space $(X,d)$, $a\in X$, $f\colon X\to[0,1]$ given by $f(x)=d(x,a)$ such that $f$ is not quotient.
\end{example}
\begin{proof}
Let $g\colon(0,1]\to[-1,1]$ be discontinuous with a connected graph.
Let $X=\{(0,0)\}\cup\{(x,xg(x))\colon x\in(0,1]\}$.
Let $d((x,y),(s,t))=\max\{|x-s|,|y-t|\}$.
Let $f\colon X\to[0,1]$ be given by $f(x)=d(x,(0,0))$.
Then $f((x,xg(x))=x$ for $x\in(0,1]$.
If $f$ were quotient, its inverse $f^{-1}$
would be continuous and then the discontinuous function $g(x)=x^{-1}\pi(f^{-1}(x))$ would be a composition of continuous functions, $\pi(x,y)=y$.
\end{proof}

\begin{example}
A connected locally compact metric space $(X,d)$, $a\in X$, $f\colon X\to[0,1]$ given by $f(x)=d(x,a)$ such that $f$ is not quotient.
\end{example}
\begin{proof}
Let $X=[0,2-1/\pi]\times\{0\}\cup\{(x,\sin(1/(2-x))\colon x\in[2-1/\pi,2)\}\cup\{2\}\times(-1,1)$.
Let $d((x,y),(s,t))=\max\{|x-s|,|y-t|\}$.
Let $f\colon X\to[0,1]$ be given by $f(x)=d(x,(0,0))$.
Then $f((x,y))=x$ for $x\in(0,1]$.
Let $G=[0,2]\times[-1,1)$. Then $G$ is open and $G=f^{-1}(f(G))$.
Since $H=\{x\in(2-1/\pi,2)\colon\sin(1/(2-x))=1\}$ is countable with $2\in\clo H$,
$f(G)=[0,2]\setminus H$ is not open in $[0,2]=f(X)$.
\end{proof}

\begin{theorem}
If $X$ is a normal space with a closed subspace that has a quotient homeomorphic to an arc, then $X$ has a quotient homeomorphic to an arc.
\end{theorem}
\begin{proof}
By assumption, there is a closed subspace $A\subset X$ and a continuous quotient surjection $g\colon A\to[0,1]$.
Since $X$ is normal, there is a continuous extension $f\colon X\to[0,1]$, $f|A=g$. We need to show that $f$ is quotient. Indeed, let $G$ be open in $X$ with $G=f^{-1}(f(G))$.
Then $f(G)=g(G\cap A)$. Indeed, let $y\in f(G)$. Since $g$ is surjective, $g(a)=y$ for some $a\in A$.
Then $f(a)=y$ and so $a\in G$. Thus $y\in g(G\cap A)$.
We still need to show that $g^{-1}g(G\cap A))=G\cap A$.
Indeed, let $a\in g^{-1}g(G\cap A))$. Then $a\in A$, $g(a)\in g(G\cap A)$ and $f(a)\in f(G)$ and thus $a\in G$.
Since $g$ is quotient, $f(G)=g(G\cap A)$ is open in $[0,1]$.
\end{proof}

\begin{corollary}
An infinite connected normal Hausdorff space with no quotient homeomorphic to an arc has to be punctiform and not locally connected.
\end{corollary}

The simplest examples of connected, not locally connected punctiform spaces are supplied by functions $f\colon\R\to\R$ with connected graphs with a dense set of discontinuity points. We show that any function with a connected graph lying on the plane collapses to an arc.

\begin{theorem}\label{spojnywykresodpada}
If $g\colon[0,1]\to[0,1]$ is a surjection with a connected graph, then $f(x,g(x))=g(x)$ is a continuous quotient surjection mapping the graph of $g$ onto $[0,1]$.
\end{theorem}
\begin{proof}
Let $G$ be an open subset of the graph of $g$ such that $f^{-1}(f(G))=G$.
Take any $y\in f(G)$, $y<1$.
Then the open subset of the graph $E=g\cap([0,1]\times(y,1])$, being neither empty nor the whole graph, cannot be a closed subset of the connected graph.
Thus we get $(x,g(x))\in\clo E\setminus E$. It follows that $g(x)=y$ and so $(x,y)\in G$.
Since $G$ is open in $g$, we get an open rectangle  $U\times V$ with $(x,y)\in(U\times V)\cap g\subset G$.
Since $(x,y)\in\clo E$, we get $(u,g(u))\in U\times V$ with $g(u)>y$.
Now, since $g$ is Darboux, both sets $g(U)$ and $V$ are intervals, and thus $[y,g(u)]\subset g(U)\cap V\subset f(G)$.
By symmetry, if $y\in f(G), y>0$, then $[a,y]\subset f(G)$ for some $a\in(0,y)$.
Thus $f(G)$ is open in $[0,1]$.
\end{proof}

\begin{example}
A function $g\colon[0,1]\to[0,1]^2$ with a connected graph such that $f(x,g(x))=f(x,(g_1(x),g_2(x)))=g_1(x)$ is not quotient.
\end{example}
\begin{proof}
Let $g(x)=(x,\sin(1/x))$ for $x\in(0,1]$ and $g(0)=(0,0)$.
That the function $f(0,(0,0))=0$ and $f(x,(x,\sin(1/x)))=x$ for $x\in(0,1]$ which maps the graph of $g$ onto $[0,1]$ is quotient is equivalent to claiming that its inverse given by  $f^{-1}(x)=(x,(x,\sin(1/x)))$ for $x\in(0,1]$ and $f^{-1}(0)=(0,(0,0))$ is continuous.
\end{proof}

\begin{theorem}
For any two distinct points of a connected graph lying on the plane there is a quotient Urysohn function.
\end{theorem}
\begin{proof}
If $g\colon[0,1]\to[0,1]$ has a connected graph,
then the graph of $h(x)=g(x)+x$ is homeomorphic to the graph of $g$.
Applying Theorem \ref{spojnywykresodpada} to $h$ instead of to $g$ makes sure that two points $(a,g(a))$ and $(b,g(b))$ with $g(a)=g(b)$ will end up in different fibers of the continuous quotient surjection onto a closed interval.
\end{proof}

\begin{theorem}
If a normal Hausdorff space contains an arc then for any two distinct points there is a quotient Urysohn function.
\end{theorem}
\begin{proof}
Let $[a,b]\subset X$ be a homeomorph of $[0,1]$ with a continuous quotient bijection $f\colon[a,b]\to[0,1]$.
Let $x,u\in X$ be any distinct points. If they both lie on the arc, then the proof is easily finished. If both lie outside the arc, then put $f(x)=0$ and $f(u)=1$ and obtain a continuous extension of $f$ to the whole $X$. It will be quotient.
If $x\in[a,b]$ and $u\not\in[a,b]$ then put $f(u)=1-f(x)$ and proceed as before.
\end{proof}

Another well-known example of a connected, not locally connected punctiform space is the Knaster-Kuratowski fan. The following proposition shows that it easily collapses to an arc.

\begin{proposition}
Let $X,Y$ be metrizable spaces.
If $X=S\cup D$ with $S$, $D$ disjoint and dense,
$Y=P\cup Q$ with $P$, $Q$ disjoint and dense,
$Z=(S\times P)\cup( D\times Q)$,
then $f\colon Z\to Y$ given by $f(x,y)=y$ is a continuous open  surjection (and thus quotient).
\end{proposition}
\begin{proof}
It is evident that $f$ is a continuous surjection.
Suppose that it is not open.
Then there is an open subset $G$ of $Z$ such that $f(G)$ is not open in $Y$.
We get a sequence $y_n\in Y\setminus f(G)$ converging to a point $y\in f(G).$
Then $f(x,y)=y$ for some $(x,y)\in G$.
Let $x_n$ be chosen so that $x_n\in B(x,1/n)$ and $x_n\in S\iff y_n\in P$.
Then $(x_n,y_n)\in Z$ converges to $(x,y)\in G$.
Since $G$ is open, we get a point $(x_k,y_k)\in G$.
Then $f(x_k,y_k)=y_k\in f(G)$.
\end{proof}

We have been interested so far in obtaining quotients homemorphic to the unit interval.
It is worth noting at this point that each connected sequential topological space can be obtained as a quotient of a connected metric space, \cite{banakh} (a topological space is sequential if each non-closed subset
contains a sequence converging to its complement; in particular, all first countable spaces are sequential).

\end{document}